\documentclass[a4paper,11pt]{amsart}
\usepackage[matrix,arrow,tips,curve]{xy}
\usepackage[english]{babel}
\usepackage{amsmath,stackrel}
\usepackage{amssymb}
\usepackage{tikz}
\usepackage{mathrsfs}
\usepackage{enumerate}
\usepackage{graphicx}
\usepackage{hyperref}
\usetikzlibrary{trees}
\usetikzlibrary{arrows}
\usepackage{xcolor}

\newtheorem{thm}{Theorem}[section]
\newtheorem{Lemma}[thm]{Lemma}
\newtheorem{Proposition}[thm]{Proposition}
\newtheorem{Corollary}[thm]{Corollary}

\newtheorem*{thm*}{Theorem}

\theoremstyle{definition}

\newtheorem{Remark}[thm]{Remark}

\newtheorem{say}[thm]{}

\newcommand{\Aut}{\operatorname{Aut}}

\newcommand{\cA}{\mathcal{A}} 
\newcommand{\cC}{\mathcal{C}} 
\newcommand{\cL}{\mathcal{L}} 
\renewcommand{\O}{\mathcal{O}} 
\renewcommand{\P}{\mathbb{P}}
\newcommand{\Z}{\mathbb{Z}}
\newcommand{\bV}{\mathbb{V}}
\newcommand{\Q}{\mathbb{Q}} 
\newcommand{\cM}{\mathcal{M}}

\DeclareMathOperator{\Pic}{Pic}

\DeclareMathOperator{\Kum}{Kum}
\DeclareMathOperator{\Jac}{Jac}

\DeclareMathOperator{\Eff}{Eff}

\DeclareMathOperator{\tr}{tr}

\DeclareMathOperator{\Sing}{Sing}

\DeclareMathOperator{\SPEnd}{SPEnd}

\DeclareMathOperator{\R}{\mathbb{R}}
\DeclareMathOperator{\C}{\mathbb{C}}

\DeclareMathOperator{\Mov}{Mov}

\hypersetup{pdfpagemode=UseNone}
\hypersetup{pdfstartview=FitH}
		
\begin{document}
\title{On automorphisms of moduli spaces of parabolic vector bundles}

\author[Carolina Araujo]{Carolina Araujo}
\address{\sc Carolina Araujo\\
IMPA, Estrada Dona Castorina 110, 22460-320 Rio de Janeiro, Brazil}
\email{caraujo@impa.br}

\author[Thiago Fassarella]{Thiago Fassarella}
\address{\sc Thiago Fassarella\\
Universidade Federal Fluminense, Rua Alexandre Moura 8 - S\~ao Domingos, 24210-200 Niter\'oi, Rio de Janeiro, Brazil}
\email{tfassarella@id.uff.br}

\author[Inder Kaur]{Inder Kaur}
\address{\sc Inder Kaur\\
IMPA, Estrada Dona Castorina 110, 22460-320 Rio de Janeiro, Brazil}
\email{inder@impa.br}

\author[Alex Massarenti]{Alex Massarenti}
\address{\sc Alex Massarenti\\ Dipartimento di Matematica e Informatica, Universit\`a di Ferrara, Via Machiavelli 30, 44121 Ferrara, Italy\newline
\indent Instituto de Matem\'atica e Estat\'istica, Universidade Federal Fluminense, Campus Gragoat\'a, Rua Alexandre Moura 8 - S\~ao Domingos\\
24210-200 Niter\'oi, Rio de Janeiro\\ Brazil}
\email{alex.massarenti@unife.it, alexmassarenti@id.uff.br}

\date{\today}

\begin{abstract}
Fix $n\geq 5$ general points $p_1, \dots, p_n\in \P^1$, and a weight vector $\cA = (a_{1}, \dots, a_{n})$ of real numbers $0 \leq a_{i} \leq 1$. 
Consider the moduli space $\mathcal{M}_{\cA}$  parametrizing rank two parabolic vector bundles with trivial determinant on $\big(\P^1, p_1,\dots , p_n\big)$ which are semistable with respect to $\cA$.
Under some conditions on the weights, we determine and give a modular interpretation for the automorphism group of the moduli space $\mathcal{M}_{\cA}$. 
It is isomorphic to $\left(\frac{\mathbb{Z}}{2\mathbb{Z}}\right)^{k}$ for some $k\in \{0,\dots, n-1\}$, and is generated by admissible elementary transformations of parabolic vector bundles. The largest of these automorphism groups, with $k=n-1$, occurs 
for the central weight $\cA_{F}= \left(\frac{1}{2},\dots,\frac{1}{2}\right)$. The corresponding moduli space ${\mathcal M}_{\cA_F}$ is a Fano variety of dimension $n-3$, which is smooth if $n$ is odd, and has isolated singularities if $n$ is even. 
\end{abstract}

\maketitle
\tableofcontents

%
%

\section{Introduction}
Let $C$ be a smooth projective curve and fix distinct points $p_1, \dots, p_n\in C$, which we refer to as parabolic points. Let $S=p_1+\cdots +p_n$ be the effective reduced divisor determined by these points. 
A \emph{quasi parabolic vector bundle} on $\big(C, S\big)$ is a vector bundle $E$ on $C$ with the additional data of a flag on the fiber over each parabolic point. If in addition we attach some weights to theses flags we call it a \emph{parabolic vector bundle}. Parabolic vector bundles were introduced by Mehta and Seshadri (\cite{Se77}, \cite{MS80}) in order to generalize to curves with cusps the Narasimhan-Seshadri correspondence between stable vector bundles on smooth projective curves and unitary representations of their fundamental groups (\cite{NS65}). 
As in the classical case, once one fixes a line bundle $L$ and a notion of slope-stability, there is a moduli space of semistable parabolic vector bundles having determinant $L$. The notion of slope-stability depends on sets of weights assigned to the parabolic flags.  Different choices of weights usually yield different moduli spaces, coming from variation of GIT.

There is one case in which the theory has been extensively investigated, and the different moduli spaces are well described. This is the case when $C\cong \P^1$ is the complex projective line, the vector bundles have rank $2$ and the flags are given by parabolic directions $V_i\subset E_{p_i}$ over each parabolic point. In this case, we may assume that the  vector bundles have trivial determinant, and the slope-stability condition depends on the choice of a weight  vector $\cA = (a_{1}, \dots, a_{n})$ of real numbers $0 \leq a_{i} \leq 1$ (see Section \ref{MPB}). We denote by $\mathcal{M}_{\cA}$ the corresponding moduli space of semistable parabolic vector bundles.

The goal of this paper is to determine and give a modular interpretation of the automorphism groups of the moduli spaces 
$\mathcal{M}_{\cA}$. Descriptions of automorphisms of moduli spaces in terms of the objects that they parametrize were obtained in many cases. See for instance \cite{BM13, Ma, MaM14, MaM, FM17, Ma17, BM17, FM18, Lin11, Ro71}, for moduli spaces of pointed curves and other configuration spaces, \cite{BGM13} for moduli spaces of vector bundles over curves, and \cite{BM16} for generalized quot schemes.

In \cite{Ba91}, Bauer described the weight polytope $\Delta\subset [0,1]^{n}$ consisting of weight vectors $\cA$ for which $\mathcal{M}_{\cA}\neq \emptyset$. 
He also exhibited a wall-and-chamber decomposition on $\Delta$ corresponding to the variation of GIT for the moduli spaces $\mathcal{M}_{\cA}$, and described the birational maps between models corresponding to different chambers. 
The weight polytope $\Delta$ is the polytope generated by the even vertices of the hypercube $[0,1]^{n}$, where the parity of
a vertex is the parity of the set of its coordinates that equal $1$. This polytope  is called \emph{demi-hypercube}. It is the weight polytope for the root system of $D_n$, and its symmetry group 
$$
\Aut(\Delta) \ \cong \ W(D_{n}) \ \cong  \ \big( \Z/2\Z\big)^{n-1} \rtimes S_{n}
$$
is generated by reflections along pairs of coordinate axes centered at the middle point $\left(\frac{1}{2},\dots, \frac{1}{2}\right)$, and 
permutations of the $n$ coordinate axes. 

\begin{say}{Elementary transformations.}\label{el}
The normal subgroup $\big( \Z/2\Z\big)^{n-1}\lhd \ \Aut(\Delta)$ of reflections admits a modular realization as 
a group of \emph{elementary transformations}, which we now describe. 
Let $\big(E, {\bf v}\big)$ be a rank $2$ quasi parabolic vector bundle on $\big( \P^1, S \big)$ of degree $0$,
and let $R\subset \{1, \dots, n\}$ be a subset of cardinality $r\geq 0$.
Identifying vector bundles with their associated locally free sheaves, we consider the natural exact sequence of sheaves
$$
0\ \to \ E' \ \stackrel{\alpha}{\to} \  E\ \to \ \bigoplus_{i\in R}(E_{p_i}/V_i)\otimes \O_{p_i} \ \to \ 0 \ .
$$
Note that we have the following equality
\[
\det E' = \det E \otimes \mathcal O_{\mathbb P^1}\big(-\sum_{i\in R}p_i\big).
\]
In particular, $E'$ is a vector bundle of rank $2$ and degree $-r$. We view $E'$ as a quasi parabolic vector bundle on $\big(\P^1, S\big)$ as follows. If $i\not\in R$, then $\alpha_{p_i}: E'_{p_i}\longrightarrow E_{p_i}$ is an isomorphism and 
$$V_i'=(\alpha_{p_i})^{-1} (V_i)\subset E'_{p_i}$$ is the parabolic direction at $p_i$. If $i\in R$, then $V_i'=\ker(\alpha_{p_i})$ is the parabolic direction at $p_i$. 
This operation corresponds to the birational transformation of ruled surfaces $\P(E)\dashrightarrow \P(E')$ obtained
by blowing-up the points $\mathbb{P}(V_i)\in\mathbb{P}(E_{p_i})$ and then blowing-down the strict transforms of the fibers $\mathbb{P}(E_{p_i})$  to the points $\mathbb{P}(V'_i)\in\mathbb{P}(E'_{p_i})$, $i\in R$. 
When $r$ is even, we obtain a correspondence 
$$
el_R \ : \ \big(E, {\bf v}\big) \ \mapsto \ \Big(E'\otimes \O_{\P^1}(r/2), \{V'_i\}\Big)
$$
between rank $2$ quasi parabolic vector bundles on $\big(\P^1, S\big)$ of degree $0$.
We call it the \emph{elementary transformation centered at the parabolic points} $\{p_i\}_{i\in R}$. 
Note that $el_R$ is not the identity unless $R=\emptyset$.
Elementary transformations are involutions and $el_R\circ el_S=el_{R\cup S\setminus R\cap S}$. 
So they form a group 
$$
\textbf{El}\ = \ \big\{ el_R \ \big| \  R\subset \{1, \dots, n\} \text{ of even cardinality}\big\}\cong\left(\frac{\mathbb{Z}}{2\mathbb{Z}}\right)^{n-1}.
$$

When we perform an elementary transformation, the stability condition is preserved after an appropriate 
modification of weights. For a weight vector $\cA = (a_{1}, \dots, a_{n})\in \Delta$ and a subset $R\subset \{1, \dots, n\}$ of even cardinality, 
we set 
\[
\cA^R := (a'_{1}, \dots, a'_{n})\in \Delta
\]
where $a'_i=a_i$ if $i\not\in R$, and $a'_i=1-a_i$ if $i\in R$. If $\big(E, {\bf v}\big)$ is semistable with respect to $\cA$, then $el_R\big(E, {\bf v}\big)$ is semistable with respect to $\cA^R$. This follows from the following observation. If  $L\subset E$ is a line bundle, then its image $L'\subset E'$ is 
$$
L' = L\otimes \mathcal O_{C}(-D) 
$$
where $D$ is the reduced divisor supported on the points $p_i$ such that $V_i\nsubseteq L$. We conclude that the correspondence $el_R$ defines an isomorphism between moduli spaces 
$$
el_R: \mathcal{M}_{\cA} \longrightarrow \mathcal M_{\cA^R}.
$$
\end{say}

The moduli space $\mathcal M_{\cA_F}$ associated to the central weight $\cA_{F} = \left(\frac{1}{2},\dots, \frac{1}{2}\right)$ is specially interesting.  It is a Fano variety of dimension $n-3$ that is smooth if $n$ is odd, and has isolated singularities if $n$ is even. 
If follows from the above discussion that $el_R$ induces an automorphism of $\mathcal M_{\cA_F}$ for every subset $R\subset \{1, \dots, n\}$ of even cardinality. In other words, we have:
$$
\left(\frac{\mathbb{Z}}{2\mathbb{Z}}\right)^{n-1}\ \cong \ \textbf{El}\ \subset\ \Aut\big(\mathcal M_{\cA_F}\big).
$$

Our first result is the following.

\begin{thm}\label{main_intro}
Fix $n\geq 5$ general points $p_1, \dots, p_n\in \P^1$ and let $\mathcal M_{\cA_F}$ be the moduli space of rank two parabolic vector bundles with trivial determinant on $\big(\P^1, S\big)$  wich are semistable with respect to the weight vector $\cA_{F} = \left(\frac{1}{2},\dots,\frac{1}{2}\right)$. Then
$$
\left(\frac{\mathbb{Z}}{2\mathbb{Z}}\right)^{n-1}\ \cong \ \emph{\textbf{El}}\ =\ \Aut\big(\mathcal M_{\cA_F}\big).
$$
\end{thm} 

We remark that for $n$ odd, the isomorphism 
$$ 
\Aut\big(\mathcal{M}_{\cA_F}\big)\cong \left(\frac{\mathbb{Z}}{2\mathbb{Z}}\right)^{n-1}
$$ 
was proved in \cite[Proposition 1.9]{AC17}, without the modular description as elementary transformations. 
For $n=5$, the moduli space $\mathcal M_{\cA_F}$ is isomorphic to a del Pezzo surface of degree four and its automorphism group is classically known (\cite[Section 8.6.4]{Do12}).

For an arbitrary weight $\cA\in \Delta$, let $\cC_\cA$ denote the subset of $\Delta$ consisting of weight vectors defining the same 
stability condition as $\cA$. It can be explicitly read off from Bauer's wall-and-chamber decomposition on $\Delta$.
Consider the subgroup of \emph{$\cA$-admissible elementary transformations}:
$$
\textbf{El}_{\cA}\ = \ \big\{ el_R \in \textbf{El} \ \big| \ \cA^R\in \cC_\cA \ \big\}\ \subset \ \textbf{El}.
$$
Then 
\stepcounter{thm}
\begin{equation}\label{El_in_Aut}
\textbf{El}_{\cA}\ \subset \ \Aut\big(\mathcal{M}_{\cA}\big). 
\end{equation}
In general one does not have equality in \eqref{El_in_Aut}. 
For instance, there are weight vectors $\cA\in \Delta$ for which $\mathcal{M}_{\cA}\cong \P^n$
(see \cite{Ba91} and Section~\ref{bir}).
However, there is an open sub-polytope of $\Delta$ for which equality in \eqref{El_in_Aut} holds. It can be described as follows. 
For every vertex $v$ of $\Delta$, let $H_v\subset \R^n$ be the hyperplane spanned by those vertices of $\Delta$ that are adjacent to $v$. Let $\Pi$ be the sub-polytope obtained from $\Delta$ by chopping off each vertex $v$ of $\Delta$ with the hyperplane $H_v$ 
(see Section~\ref{bir}). It contains in its interior the subset $\cC_{\cA_F}$ consisting of weight vectors defining the same stability condition as $\cA_F$.

\begin{Corollary}\label{allweights}
Fix $n\geq 5$ general points $p_1, \dots, p_n\in \P^1$ and let $\cA$ be a weight vector in the interior of the polytope $\Pi$ defined above. Let $\mathcal{M}_{\cA}$ be the moduli space of rank two parabolic vector bundles with  trivial determinant on $\big(\P^1, S\big)$  which are semistable with respect to the weight vector $\cA$. Then
$$
{\emph{\textbf{El}}}_{\cA}\ = \ \Aut\big(\mathcal{M}_{\cA}\big). 
$$
\end{Corollary}

The polytope $\Pi$ has a natural description from the point of view of birational geometry. 
Namely, for weights $\cA$ in the interior of the polytope $\Pi$, the moduli spaces $\mathcal{M}_{\cA}$  are small modifications of the Fano variety $\mathcal M_{\cA_F}$.

\medskip

This paper is organized as follows. 
In Section~\ref{MPB}, we revise basic properties of moduli spaces of parabolic vector bundles, Hitchin systems and spectral curves.
This theory is used in Section~\ref{main_section}  to prove Theorem~\ref{main_intro}. 
In Section~\ref{bir}, we use birational geometry and the small equivalence of models to prove Corollary~\ref{allweights}.
In Section \ref{hyp}, we describe the automorphism group of moduli spaces of involutional vector bundles on hyperelliptic curves by relating them to $\mathcal{M}_{\cA}$.
\medskip

\noindent {\bf Acknowledgements.}
Carolina Araujo was partially supported by CNPq and Faperj Research Fellowships. Thiago Fassarella was partially supported by CNPq. Inder Kaur was supported by a CNPq post-doctoral fellowship. Alex Massarenti is a member of the Gruppo Nazionale per le Strutture Algebriche, Geometriche e le loro Applicazioni of the Istituto Nazionale di Alta Matematica "F. Severi" (GNSAGA-INDAM). Part of this work was developed during the visit of some of the authors to ICTP, funded by Carolina Araujo's ICTP Simons Associateship. We thank ICTP and Simons Foundation for the great working conditions and the financial support.

%
%

\section{Moduli spaces of parabolic vector bundles on $\P^1$}\label{MPB}

Fix  $p_1, \dots, p_n\in \P^1$ general points and denote by $S=p_1+\cdots +p_n$ the effective reduced divisor determined by them.

\begin{say}{\bf Quasi parabolic vector bundles.}
A \textit{quasi parabolic vector bundle} $(E, {\bf v})$, ${\bf v} = \{V_{i}\}$,  of rank two  on $\big(\P^1, S\big)$ consists  of
\begin{itemize}
\item[-] a vector bundle $E$ of rank two on $\P^{1}$; and
\item[-] for each  $i = 1,\dots,n$, a $1$-dimensional linear subspace $V_{i} \subset E_{p_{i}}$.
\end{itemize}
By abuse of notation we often write $E$ for $(E, {\bf v})$. 
We refer to the points $p_i$'s as parabolic points, and to the supspace $V_{i} \subset E_{p_{i}}$ as the parabolic direction of $E$ at $p_i$. 

Let $(E, {\bf v})$ and $(E', {\bf v}')$ be quasi parabolic vector bundles. A homomorphism of vector bundles $f : E \longrightarrow E'$ is called \textit{parabolic} if $f(V_{i}) \subseteq V'_i$ for every $i = 1,\dots,n$. 
It is called \textit{strongly parabolic} if $f(E_{p_i}) \subseteq V'_i$ and $f(V_i)=0$ for every $i = 1,\dots,n$. 
We denote by $\mathcal{PH}om(E,E')$ and $\mathcal{SPH}om(E,E')$ the sheaves of parabolic and strongly parabolic homomorphisms, by $\mathcal{PE}nd(E):=\mathcal{PH}om(E,E)$ and $\mathcal{SPE}nd(E):=\mathcal{SPH}om(E,E)$ the sheaves of parabolic and strongly parabolic endomorphisms of $(E, {\bf v})$, and by $\mathcal{PE}nd_0(E)$ and $\mathcal{SPE}nd_0(E)$ their subsheaves of traceless endomorphisms.

By taking the trace of the product of two endomorphisms, one defines symmetric $\O_{\P^1}$-bilinear  sheaf homomorphisms
$$
\mathcal{E}nd(E)\times \mathcal{E}nd(E) \to \O_{\P^1} \ \ \text{and} \ \ \mathcal{E}nd_0(E)\times \mathcal{E}nd_0(E) \to \O_{\P^1}.
$$
A simple linear algebra computation then yields the following parabolic dualities
\begin{equation}\label{duality}
\mathcal{PE}nd(E)^{\vee} \ \cong \ \mathcal{SPE}nd(E)\otimes \O_{\P^1}(S)  \ \ \text{and} \ \  
\mathcal{PE}nd_0(E)^{\vee} \ \cong \ \mathcal{SPE}nd_0(E)\otimes \O_{\P^1}(S).
\end{equation}
\end{say}

\begin{say}{\bf Weights and stability conditions.}
Fix a weight  vector $\cA = (a_{1}, \dots, a_{n})$ of real numbers $0 \leq a_{i} \leq 1$.
The  \textit{parabolic slope}  of $(E, {\bf v})$ with respect to $\cA$ is 
$$
\mu_{\mathcal{A}}(E) = \frac{\deg E + \sum_{i=1}^{n}a_{i}}{2}.
$$ 
Let $L \subset E$ be a line subbundle. For each  $i = 1,\dots,n$, set 
$$a_i(L,E)\ \ = \  \left\{ 
\begin{aligned}
& a_i \ & \text{ if } L_{p_{i}} = V_{i},\\
&0 \ &  \text{ if } L_{p_{i}} \neq V_{i}.
\end{aligned}
\right.$$
The  \textit{parabolic slope}  of $L \subset E$ with respect to $\cA$ is 
$$
\mu_{\mathcal{A}}(L,E) =  \deg(L)+\sum_{i=1}^{n}a_i(L,E).
$$

A quasi parabolic vector bundle $(E,{\bf v})$  is $\mu_{\mathcal{A}}$-\textit{semistable} (respectively $\mu_{\mathcal{A}}$-\textit{stable}) if for every  line subbundle $L \subset E$ we have $\mu_{\mathcal{A}}(L,E) \leq  \mu_{\mathcal{A}}(E)$ (respectively $\mu_{\mathcal{A}}(L,E) < \mu_{\mathcal{A}}(E)$). A \textit{parabolic vector bundle} is a quasi parabolic vector bundle together with a  weight vector $\cA$. We say that a parabolic vector bundle is semistable if the corresponding quasi parabolic vector bundle is $\mu_{\mathcal{A}}$-\textit{semistable}. 
\end{say}

By \cite{MS80}, for each fixed degree $d\in \Z$, there is a moduli space $\mathcal{M}_{\cA}(d)$ parametrizing rank two degree $d$ quasi parabolic vector bundles on $\big(\P^1, S\big)$ which are $\mu_{\mathcal{A}}$-semistable. It is a normal projective variety.
By twisting vector bundles with a fixed line bundle, we see that $\mathcal{M}_{\cA}(d) \cong \mathcal{M}_{\cA}(d')$ whenever $d$ and $d'$ have the same parity. 
By performing an elementary transformation centered at one parabolic point $p_i$, as described in the introduction, we see that $\mathcal{M}_{\cA}(d) \cong \mathcal M_{\cA^{i}}(d-1)$, where  
$$
\cA^{i}=(a_1, \dots, 1-a_i, \dots, a_n).
$$ 
So from now on we assume that $d=0$ and write simply $\mathcal{M}_{\cA}$ for the corresponding moduli space.  

Let $\mathcal{M}^{s}_{\cA}\subset\mathcal{M}_{\cA}$ be the Zariski open subset parametrizing stable parabolic vector bundles. 
If it is not empty, then it is an irreducible smooth quasi-projective variety of dimension $n-3$. 
We describe the tangent space of $\mathcal{M}_{\cA}^{s}$ at a point $(E, {\bf v}))$. We denote by $\omega_{\mathbb P^1}$ the canonical sheaf of $\mathbb P^1$. For any invertible sheaf $\mathcal L$ and $S\in {\rm Div}(\mathbb P^1)$ we write $\mathcal L (S)$ instead of $\mathcal L\otimes\mathcal{O}_{\mathbb{P}^1}(S)$. We also write $T_{E}\mathcal{M}_{\cA}^{s}$ for the tangent space of $\mathcal{M}_{\cA}^{s}$ at $(E, {\bf v})$.
By \cite[Theorem 2.4]{Yo95},
\begin{equation}\label{tan}
 T_{E}\mathcal{M}_{\cA}^{s} \cong {\rm H}^1(\mathbb{P}^1,\mathcal{PE}nd(E))\cong {\rm H}^0(\mathbb{P}^1,\mathcal{SPE}nd(E)\otimes \omega_{\mathbb P^1}(S))^{\vee},
\end{equation}
where the second isomorphism holds by \eqref{duality} and Serre duality. We refer to the combination of these two dualities as 
\emph{parabolic Serre duality}.

Let $\theta \in {\rm H}^0(\mathbb{P}^1,\mathcal{SPE}nd(E)\otimes \omega_{\mathbb P^1}(S))$ be a global section. 
For each parabolic point $p_i$, the residual endomorphism ${\rm Res}(\theta, p_i)\in {\rm End}(E_{p_i})$ is well defined.    
The strongly parabolic condition implies that these endomorphisms are nilpotent for each parabolic point $p_i$.
In particular the trace $\tr(\theta)\in {\rm H}^0(\mathbb{P}^1, \omega_{\mathbb P^1}(S))$ of  $\theta$ vanishes at $p_1, \dots, p_n$, and thus $\tr(\theta)=0$. So we have an isomorphism
$$
T_{E}\mathcal{M}_{\cA}^{s} \cong {\rm H}^0(\mathbb{P}^1,\mathcal{SPE}nd_0(E)\otimes \omega_{\mathbb P^1}(S))^{\vee}.
$$

\begin{say}{\bf Parabolic Higgs bundles.}\label{higgs}
Given a parabolic vector bundle $\big(E, {\bf v}\big)$ on $\big(\P^1, S\big)$, a
\textit{Higgs field} on $(E, {\bf v})$ is a section
$$
\theta \in {\rm H}^0(\mathbb{P}^1,\mathcal{SPE}nd(E)\otimes \omega_{\mathbb P^1}(S)).
$$
In order to simplify notation we shall denote the vector space above by 
\[
{\it Higgs}(E,{\bf v}):= {\rm H}^0(\mathbb{P}^1,\mathcal{SPE}nd(E)\otimes \omega_{\mathbb P^1}(S)).
\]
In view of (\ref{tan}), there is an  isomorphism 
\[
{\it Higgs}(E,{\bf v}) \cong  T_{E}^*\mathcal{M}_{\cA}^{s}
\]
for each $(E,{\bf v}) \in \mathcal{M}_{\cA}^{s}$. 
As we noted above, the trace of a Higgs field  vanishes.
This implies that the minimal polynomial of $\theta$ is $t^2 + \det(\theta)$. 

A \emph{parabolic Higgs bundle} $(E,\theta)$ on $\big(\P^1, S\big)$ consists of a parabolic vector bundle $\big(E, {\bf v}\big)$
together with a Higgs field $\theta$ on $E$. It is $\mu_{\mathcal{A}}$-\textit{semistable} (respectively $\mu_{\mathcal{A}}$-\textit{stable}) if for every  line subbundle $L \subset E$ invariant under $\theta$, we have $\mu_{\mathcal{A}}(L,E) \leq  \mu_{\mathcal{A}}(E)$
(respectively $\mu_{\mathcal{A}}(L,E) <  \mu_{\mathcal{A}}(E)$).

We denote by $\mathcal{M}_{\cA}^{Higgs}$ the moduli space of $\mu_{\mathcal{A}}$-semistable parabolic Higgs bundles of rank two and  trivial determinant. It is a normal, quasiprojective variety of dimension $2n-6$. By \eqref{tan}, $\mathcal{M}_{\cA}^{Higgs}$ contains as an open subset the total space of the cotangent bundle of $\mathcal{M}_{\cA}^{s}$.
\end{say}

\begin{say}{\bf The Hitchin map.}\label{Hitchin}
Let $(E,\theta)$ be a parabolic Higgs bundle on $\big(\P^1, S\big)$, and consider 
$\det(\theta)\in {\rm H}^0(\mathbb{P}^1,\omega_{\mathbb{P}^1}^{\otimes 2}(2S))$.
Since ${\rm Res}(\theta, p_i)$ is nilpotent for every parabolic point $p_i\in\P^1$, $\det(\theta)$ lies in the linear subspace 
$V\subset {\rm H}^0(\mathbb{P}^1,\omega_{\mathbb{P}^1}^{\otimes 2}(2S))$ consisting of sections 
 vanishing at $p_1, \dots, p_n$. 
Identifying $V$ with ${\rm H}^0(\mathbb{P}^1,\omega_{\mathbb{P}^1}^{\otimes 2}(S))$, the \textit{Hitchin map} is defined as
$$
\begin{array}{cccc}
H: &\mathcal{M}_{\cA}^{Higgs}& \longrightarrow & {\rm H}^0(\mathbb{P}^1,\omega_{\mathbb{P}^1}^{\otimes 2}(S))\\
      & (E, {\bf v},\theta) & \longmapsto & \det(\theta).
\end{array}
$$
\end{say}

Our next goal is to describe the fibers of the Hitchin map, and of its restriction to the total space of the cotangent bundle of $\mathcal{M}_{\cA}^{s}$, which we denote by
$$
h: T^{*}\mathcal{M}_{\cA}^{s}\longrightarrow {\rm H}^0(\mathbb{P}^1,\omega_{\mathbb{P}^1}^{\otimes 2}(S)).
$$
For this purpose we recall the properties of spectral curves associated to the Hitchin map.

\begin{say}{\bf Spectral curves.}\label{spectral}
Denote by $\bV$ the total space of the sheaf $\omega_{\mathbb{P}^1}(S)$, with natural map
$\pi: \bV\longrightarrow \P^1$. There is a tautological section $s\in {\rm H}^0\big(\bV, \pi^*(\omega_{\mathbb{P}^1}(S))\big)$.
Given $a\in V\subset {\rm H}^0(\mathbb{P}^1,\omega_{\mathbb{P}^1}^{\otimes 2}(2S))$, we define 
the spectral curve $C_a$ associated to $a$ as the zero locus of the section 
$$
s^2+\pi^*a\in  {\rm H}^0\big(\bV, \pi^*(\omega^{\otimes 2}_{\mathbb{P}^1}(2S))\big).
$$ 
We denote by $\pi_a: C_a\longrightarrow\P^1$ the restriction of $\pi$ to $C_a$. 
It is a $2:1$ map branched over the zero locus of the global section $a$.
\end{say}

\begin{say}{\bf The Fano model.}
The central weight vector $\cA_{F}= \left(\frac{1}{2},\dots,\frac{1}{2}\right)$ yields a distinguished moduli space $\mathcal M_{\cA_F}$.
For $n=4$, we have $\mathcal M_{\cA_F}\cong \P^1$. So from now on we assume that $n\geq 5$. 

The moduli space $\mathcal M_{\cA_F}$  is a  Fano variety of dimension $n-3$ (see \cite{Mu05, Casagrande, AM16}). 
If $n$ is odd, then  there are no stricly $\mu_{\mathcal{A}}$-semistable bundles and so $\mathcal M_{\cA_F}$ is smooth.
If $n$ is even, then 
$$
\Sing\big(\mathcal M_{\cA_F}\big) = \mathcal M_{\cA_F}\setminus \mathcal{M}_{\cA_F}^s
$$ 
consists of a finite set of points (see \cite[Section 2]{BHK10}). 
\end{say}

We summarize in the following proposition the description of the  fibers of the Hitchin map 
$$H: \mathcal{M}_{\cA_F}^{Higgs} \longrightarrow  {\rm H}^0(\mathbb{P}^1,\omega_{\mathbb{P}^1}^{\otimes 2}(S))$$
in terms of spectral curves.

\begin{Proposition}[{\cite[Section 2, Proposition 2.2, Lemma 3.1]{BHK10}}]\label{SpecLem}
Let the notation be as above and fix a general section $a\in  {\rm H}^0(\mathbb{P}^1,\omega_{\mathbb{P}^1}^{\otimes 2}(S))$. Then
\begin{itemize}
\item[(i)] The spectral curve $C_a$ is a smooth and connected curve of genus $n-3$.
\item[(ii)] The fiber $H^{-1}(a)$ is an abelian variety isomorphic to $\Pic^{n-2}(C_a)$.
\item[(iii)] The codimension of $H^{-1}(a)\setminus h^{-1}(a)$ in $H^{-1}(a)$ is at least two.
\item[(iv)] Denote by $p: h^{-1}(a)\longrightarrow \mathcal{M}_{\cA_F}^s$ the restriction of the natural projection 
$T^{*}\mathcal{M}_{\cA_F}^s\longrightarrow \mathcal{M}_{\cA_F}^s$, and by $\Theta$ the theta divisor on $\Pic^{n-2}(C_a)\supset h^{-1}(a)$. 
Then 
$$p^{*}\big(-K_{\mathcal{M}_{\cA_F}^s}\big) \ = \ 4^{n-3}\Theta_{|h^{-1}(a)} \ .$$
\end{itemize}
\end{Proposition}

\begin{Remark}\label{affinization} 
It follows from Proposition~\ref{SpecLem} that the Hitchin map  $
h: T^{*}\mathcal{M}_{\cA_F}^s\longrightarrow {\rm H}^0(\mathbb{P}^1,\omega_{\mathbb{P}^1}^{\otimes 2}(S))
$
is the affinization of $T^{*}\mathcal{M}_{\cA_F}^s$. In other words, viewed as an affine variety, ${\rm H}^0(\mathbb{P}^1,\omega_{\mathbb{P}^1}^{\otimes 2}(S))$ is the spectrum of the ring of regular functions on $T^{*}\mathcal{M}_{\cA_F}^s$.\end{Remark}

\begin{say}{\bf The natural involution on the fibers of the Hitchin map.}\label{involution}
The natural involution $i_a: C_a\longrightarrow C_a$ switching the sheets of the $2:1$ covering $\pi_a: C_a\longrightarrow\mathbb{P}^1$ induces the involution on $\Pic^{n-2}(C_a)$ mapping $\mathcal{L}\in\Pic^{n-2}(C_a)$ to $i_{a}^{*}\mathcal{L}$.
We want to describe the corresponding involution on the fiber $H^{-1}(a)\subset  \mathcal{M}_{\cA_F}^{Higgs}$. 

For this purpose, let us review the correspondence in Proposition~\ref{SpecLem}(ii). 
Given a line bundle $\mathcal{L}\in \Pic^{n-2}(C_s)$, we consider the rank 2 vector bundle  $E=(\pi_a)_{*}\mathcal{L}$ on $\P^1$.
The parabolic points $p_1,\dots,p_n\in\P^1$ are contained in the ramification locus of $\pi_a$.
Therefore there is a distinguished $1$-dimensional linear subspace $V_i$ in the fiber $E_{p_i}$.
The tautological section 
$$
s_a=s_{|C_a}\in {\rm H}^0\big(C_a, \pi_a^*(\omega_{\mathbb{P}^1}(S))\big)
$$ 
induces a homomorphism $\theta={(s_a)}_{*}: E\longrightarrow E\otimes \omega_{\mathbb{P}^1}(S)$. 
The parabolic Higgs bundle  on $\big(\P^1, S\big)$ associated to the line bundle $\mathcal{L}$ is $(E, {\bf v}, \theta)$.
The equation $E=(\pi_a)_{*}\mathcal{L}$ can be viewed as the eigenspace decomposition of $\theta$ on $E$. 

Now notice that the parabolic Higgs bundle  on $\big(\P^1, S\big)$ associated to the line bundle $i_{a}^{*}\mathcal{L}$ is 
$(E, {\bf v}, \theta')$, where $\theta'$ is obtained from $\theta$ by swapping the eigenspaces. Since $\theta$ is traceless, its eigenvalues $\lambda_1,\lambda_2$ satisfy $\lambda_1 = -\lambda_2$. Hence $\theta'=-\theta$.

We conclude that the involution on the fiber $H^{-1}(a)\subset  \mathcal{M}_{\cA_F}^{Higgs}$ induced by the natural involution $i_a: C_a\longrightarrow C_a$ maps $(E, {\bf v},\theta)$ to $(E, {\bf v},-\theta)$.
\end{say}

%
%

\section{{The automorphism group of the Fano model {$\mathcal M_{\cA_F}$}}} \label{main_section} 

In this section we show that the automorphism group of the Fano model {$\mathcal M_{\cA_F}$} is the group of elementary transformations 
$\textbf{El}\cong\left(\frac{\mathbb{Z}}{2\mathbb{Z}}\right)^{n-1}$ (Theorem~\ref{main_intro}).

Let  $\varphi\in \Aut(\mathcal{M}_{\cA_{F}})$ be an automorphism sending a general rank two parabolic vector bundle
$(E, {\bf v})$ to  $(E', {\bf v}')$. 
Since $\textbf{El}\subset \Aut(\mathcal{M}_{\cA_{F}})$ is finite, in order to prove that the groups coincide, it is enough to show that there is 
an elementary transformation $el_R\in \textbf{El}$ as defined in Paragraph~\ref{el} sending $(E, {\bf v})$ to  $(E', {\bf v}')$.
This is equivalent to showing that the blowup of $\P(E)$ at the finite set of points $\{\P(V_i)\}_{i = 1,\dots,n}$  is isomorphic over $\P^1$ to the blowup of $\P(E')$ at  $\{\P(V'_i)\}_{i = 1,\dots,n}$.
In order to prove this isomorphism, we first show how to recover the blowup of $\P(E)$ at $\{\P(V_i)\}_{i = 1,\dots,n}$ as the projectivization of the \emph{nilpotent cone} associated to $E$. 
This construction works for any smooth projective curve $C$.

\begin{say}{\bf The nilpotent cone.}\label{nil}
Let $C$ be a smooth projective curve and fix parabolic points $p_1, \dots, p_n\in C$. 
Let $\big(E, {\bf v}\big)$ be a rank $2$ quasi parabolic vector bundle on $\big(C, S\big)$. 
For any invertible sheaf $\cL$ on $C$, we consider the locally free subsheaf of $\mathcal{SPE}nd(E)\otimes \cL$ consisting of traceless endomorphisms.
We denote this sheaf by $\mathcal{SPE}nd_0(E)\otimes \cL$, and the corresponding vector bundle on $C$ by  $\SPEnd_0(E,\cL)$.
Notice that their rank is $3$.
We will define a codimension one quadratic cone bundle  ${N}_E\subset \SPEnd_0(E,\cL)$, the \emph{nilpotent cone of $E$}.

For any $p\in C\setminus \{p_1, \dots, p_n\}$ consider the cone of $\SPEnd_0(E,\cL)_p\cong\C^3$ consiting of nilpotent elements:
$${N}_{E,p} \ = \  \{\Phi\in \SPEnd_0(E,\cL)_p\ | \ \Phi^2 = 0\}\ \subseteq \ \SPEnd_0(E,\cL)_p.$$
Note that since the endomorphisms $\Phi\in \SPEnd_0(E,\cL)_p$ are traceless, the condition $\Phi^2 = 0$ is equivalent to $\det(\Phi) = 0$.
By letting $p$ vary in $C\setminus \{p_1, \dots, p_n\}$, we get a cone bundle in $\SPEnd_0(E,\cL)$ over $C\setminus \{p_1, \dots, p_n\}$.
We define the nilpotent cone ${N}_E$ as the closure of this cone bundle in $\SPEnd_0(E,\cL)$.
\end{say}

\begin{Proposition}\label{NCbu}
Let the notation be as above. 
Then the projectivized nilpotent cone $\mathbb{P}({N}_{E})$ is isomorphic over $C$ to the blow-up of the ruled surface $\mathbb{P}(E)$ at the set of points $\{\P(V_i)\}_{i = 1,\dots,n}$. 
\end{Proposition}

\begin{proof}
Let $U\subseteq C$ be a trivializing open subset for both $E$ and $\SPEnd_0(E,\cL)$ containing only one of the parabolic points, $p_i\in U$.
We fix  an identification $E_{|U}\cong U\times\mathbb{C}^2$ and 
basis for $\C^2$ with respect to which the parabolic direction at $p_i$ is $V_i = \left\langle (1,0)\right\rangle$.

Write $t$ for a local parameter for $U$ at $p_i$. 
After shrinking $U$ if necessary, we may assume that $t$ is a regular function on $U$. 
Sections of $\SPEnd_0(E,\cL)$ over $U$ are families of endomorphisms given by matrices of the form
$$
M_t=\left(\begin{array}{cc}
t\alpha & \beta \\ 
t\gamma & -t\alpha
\end{array}\right), 
$$
with $\alpha, \beta, \gamma \in \Gamma(U, \cL)$. 
So we can fix an identification $\SPEnd_0(E,\cL)_{|U}\cong U\times\mathbb{C}^3$ and 
basis for $\C^3$ with respect to which the endomorphism of $E_p$ corresponding to a point 
\[
\big(p,(a,b,c)\big)\in (U\setminus\{p_i\})\times\mathbb{C}^3
\] 
is given by the matrix 
$
\left(\begin{array}{cc}
t(p)a & b \\ 
t(p)c & -t(p)a
\end{array}\right).
$
We have 
$$
\det(M_t)= -t(t\alpha^2+\beta\gamma).
$$ 
So we see that in $\SPEnd_0(E,\cL)_{|U}\cong U\times\mathbb{C}^3$, the nilpotent cone ${N}_{E}$ is cut out by the equation 
$$ta^2+bc=0.$$ 
This shows that $\mathbb{P}({N}_{E})_{|U}\subset U\times\P^2 $ is a smooth surface, the fibers of $\mathbb{P}({N}_{E})_{|U}\longrightarrow U$ over $U\setminus\{p_i\}$
are smooth conics, and the fiber over $p_i$ is the union of two intersecting lines 
\[
F_1=\{b=0\}\;\;\; \text{and}\;\;\; F_2=\{c=0\}.
\] 
From the defining equation of $\mathbb{P}({N}_{E})_{|U}$, we see that $(a:c) = (b:-ta)$, and so we have a morphism 
$$
\begin{array}{cccc}
f_{U}: &\mathbb{P}({N}_{E})_{|U} & \longrightarrow & \mathbb{P}(E)_{|U}\\
      & \big(p,(a,b,c)\big) & \longmapsto & \big(p,(a:c)\big)
\end{array}
$$
mapping $F_1$ isomorphically onto the fiber of $\mathbb{P}(E)_{|U}\longrightarrow U$ over $p_i$, and contracting $F_2$ to the point $\P(V_i)\in \mathbb{P}(E)_{|U}$. 

On $U\setminus\{p_i\}$, the vector $(a,c)$ is precisely the eigenvector of the nilpotent matrix 
\[
\left(\begin{array}{cc}
ta & b \\ 
tc & -ta
\end{array}\right).
\]
Therefore the local morphisms $f_U$ glue together to define a global birational morphism over $C$
$$f: \mathbb{P}({N}_{E})\longrightarrow\mathbb{P}(E).$$    
It is an isomorphism away from $n$ smooth rational curves, which get contracted to the points $\P(V_i)\in \mathbb{P}(E)$, $i = 1,\dots,n$,
and the result follows. 
\end{proof}

Now we go back to our original setting, with $C\cong \P^1$. 
In the proof of Theorem~\ref{main_intro} we will apply Proposition~\ref{NCbu} with 
$\cL=\omega_{\mathbb{P}^1}(S)$.
We will need the following result.

\begin{Lemma}\label{globgen}
Suppose that $n\geq 6$, and let $(E,{\bf v})\in {\mathcal{M}}_{\cA_F}$ be a general parabolic vector bundle. Then 
$\mathcal{SPE}nd_0(E)\otimes \omega_{\mathbb{P}^1}(S)$ is globally generated.
\end{Lemma}
\begin{proof}
For any point $p\in\mathbb{P}^1$, evaluation at $p$ yields an exact sequence: 
$$
0\rightarrow \mathcal{SPE}nd_0(E)\otimes \omega_{\mathbb{P}^1}(S-p)\rightarrow \mathcal{SPE}nd_0(E)\otimes \omega_{\mathbb{P}^1}(S)\rightarrow\mathcal{SPE}nd_0(E)\otimes \omega_{\mathbb{P}^1}(S)_{p}\rightarrow 0.
$$
By parabolic Serre duality, 
$$
{\rm H}^1(\mathbb{P}^1,\mathcal{SPE}nd_0(E)\otimes \omega_{\mathbb{P}^1}(S-p))\cong {\rm H}^0(\mathbb{P}^1,\mathcal{PE}nd_0(E)\otimes \mathcal{O}_{\mathbb{P}^1}(p))^{\vee}.
$$
So, in order to show that $\mathcal{SPE}nd_0(E)\otimes \omega_{\mathbb{P}^1}(S)$ is globally generated, it is enough to show that 
\stepcounter{thm}
\begin{equation}\label{go}
{\rm H}^0(\mathbb{P}^1,\mathcal{PE}nd_0(E)\otimes \mathcal{O}_{\mathbb{P}^1}(1))=\{0\}.
\end{equation}
Since $(E,{\bf v})\in {\mathcal{M}}_{\cA_F}$ is general, the underlying vector bundle $E$ is free,
and a global section in ${\rm H}^0(\mathbb{P}^1,\mathcal{E}nd_0(E)\otimes \mathcal{O}_{\mathbb{P}^1}(1))$ can be represented by a traceless $2\times 2$ matrix of linear forms on $\mathbb{P}^1$. The vector space of such matrices has dimension $6$. 
Each parabolic condition $\phi(V_i)\subseteq V_i$ imposes one linear condition. 
A straightforward computation shows that, since the parabolic directions are general, we get $n$ linearly independent conditions. Therefore, (\ref{go}) holds for $n\geq 6$. 
\end{proof}

\begin{Remark}\label{NE_from_h}
It follows from Lemma~\ref{globgen} that there is a surjective map of vector bundles on $\P^1$
$$
{\it Higgs}(E,{\bf v})\times \P^1
\ \stackrel{\alpha}{\to} \ \SPEnd_0\big(E,\omega_{\mathbb{P}^1}(S)\big).
$$
By identifying ${\it Higgs}(E,{\bf v})$
with the cotangent space $T^{*}_{E}\mathcal{M}_{\cA_F}^s$, we describe the quadratic cone $\alpha^{-1}({N}_{E})$ in terms of the restriction of the Hitchin map
$$
h_E=h_{|T^{*}_{E}\mathcal{M}_{\cA}^{s}}: \ T^{*}_{E}\mathcal{M}_{\cA}^{s} \ \longrightarrow \ {\rm H}^0(\mathbb{P}^1,\omega_{\mathbb{P}^1}^{\otimes 2}(S)).
$$
Given a point $p\in \P^1\setminus \{p_1, \dots, p_n\}$, let $V_p\subset {\rm H}^0(\mathbb{P}^1,\omega_{\mathbb{P}^1}^{\otimes 2}(S))$ be the linear space consisting of sections vanishing at $p$. Then 
$$
\alpha^{-1}({N}_{E})_p \ = \ h_E^{-1}(V_p).
$$
Working with a trivialization of $\SPEnd_0\big(E,\omega_{\mathbb{P}^1}(S)\big)$ in a neighborhood of $p$, 
as in the proof of Proposition~\ref{NCbu}  above, we see that the vertex of the cone $\alpha^{-1}({N}_{E})_p$ is a codimension $3$ linear subspace of ${\it Higgs}(E,{\bf v})$ that coincides with the kernel of $\alpha_p$.
\end{Remark}

\begin{proof}[{Proof of Theorem~\ref{main_intro}}]
Let  $\varphi\in \Aut(\mathcal{M}_{\cA_{F}})$ be an automorphism, and consider the induced homomorphism on the cotangent bundle
$$
{d\varphi}: \ T^{*}\mathcal{M}_{\cA_F}^s \ \longrightarrow \ T^{*}\mathcal{M}_{\cA_F}^s.
$$
Recall from Remark~\ref{affinization} that the Hitchin map  $
h: T^{*}\mathcal{M}_{\cA_F}^s\longrightarrow {\rm H}^0(\mathbb{P}^1,\omega_{\mathbb{P}^1}^{\otimes 2}(S))
$
is the affinization of $T^{*}\mathcal{M}_{\cA_F}^s$. Therefore there is morphism of affine varieties 
$$
f: \ {\rm H}^0(\mathbb{P}^1,\omega_{\mathbb{P}^1}^{\otimes 2}(S)) \ \longrightarrow \ {\rm H}^0(\mathbb{P}^1,\omega_{\mathbb{P}^1}^{\otimes 2}(S))
$$
making the following diagram commute: 
  \[
  \begin{tikzpicture}[xscale=4.5,yscale=-1.5]
    \node (A0_0) at (0, 0) {$T^{*}\mathcal{M}_{\cA_F}^s$};
    \node (A0_1) at (1, 0) {$T^{*}\mathcal{M}_{\cA_F}^s$};
    \node (A1_0) at (0, 1) {${\rm H}^0(\mathbb{P}^1,\omega_{\mathbb{P}^1}^{\otimes 2}(S))$};
    \node (A1_1) at (1, 1) {${\rm H}^0(\mathbb{P}^1,\omega_{\mathbb{P}^1}^{\otimes 2}(S)).$};
    \path (A0_0) edge [->]node [auto] {$\scriptstyle{d\varphi}$} (A0_1);
    \path (A1_0) edge [->]node [auto] {$\scriptstyle{f}$} (A1_1);
    \path (A0_1) edge [->]node [auto] {$\scriptstyle{h}$} (A1_1);
    \path (A0_0) edge [->,swap]node [auto] {$\scriptstyle{h}$} (A1_0);
  \end{tikzpicture}
  \]
We will show that the map $f$ is multiplication by a nonzero constant.

The $\mathbb{C}^{*}$-action by dilations on the fibers of the map $T^{*}\mathcal{M}_{\cA_F}^s\longrightarrow \mathcal{M}_{\cA_F}^s$ induces the $\mathbb{C}^{*}$-action on ${\rm H}^0(\mathbb{P}^1,\omega_{\mathbb{P}^1}^{\otimes 2}(S))$ given by $t\cdot a = t^2a$. Since $d\varphi$ is $\mathbb{C}^{*}$-equivariant, so is $f$. This implies that $f$ sends lines through the origin to lines through the origin, and hence $f$ is linear.

Consider a general section $a\in {\rm H}^0(\mathbb{P}^1,\omega_{\mathbb{P}^1}^{\otimes 2}(S))$, and set
$a' = f(a)$. 
By Proposition~\ref{SpecLem}, the spectral curves $C_a$ and $C_{a'}$ are smooth, and the isomorphism $d\varphi_{|h^{-1}(a)}:h^{-1}(a)\longrightarrow h^{-1}(a')$ extends to an isomorphism of polarized abelian varieties
$$
F: \ \Pic^{n-2}(C_a) \cong H^{-1}(a) \longrightarrow H^{-1}(a')\ \cong \ \Pic^{n-2}(C_{a'}).
$$
By Paragraph~\ref{involution}, the isomorphism $F$ commutes with the involutions on $\Pic^{n-2}(C_a)$ and $\Pic^{n-2}(C_{a'})$ induced by the natural involutions on $C_a$ and $C_{a'}$. 

Torelli theorem implies that $F$ comes from an isomorphism $F_*:C_a \longrightarrow C_{a'}$ between the spectral curves. Moreover, since $F$ commutes with the involutions on $\Pic^{n-2}(C_a)$ and $\Pic^{n-2}(C_{a'})$ induced by the natural involutions on $C_a$ and $C_{a'}$, we have a commutative diagram
\[
  \begin{tikzpicture}[xscale=2.9,yscale=-1.5]
    \node (A0_0) at (0, 0) {$C_{a}$};
    \node (A0_1) at (1, 0) {$C_{a'}$};
    \node (A1_0) at (0, 1) {$\mathbb{P}^1$};
    \node (A1_1) at (1, 1) {$\mathbb{P}^1$};
    \path (A0_0) edge [->]node [auto] {$\scriptstyle{F_*}$} (A0_1);
    \path (A1_0) edge [->]node [auto] {$\scriptstyle{\overline{F_*}}$} (A1_1);
    \path (A0_1) edge [->]node [auto] {$\scriptstyle{\pi_{a'}}$} (A1_1);
    \path (A0_0) edge [->,swap]node [auto] {$\scriptstyle{\pi_{a}}$} (A1_0);
  \end{tikzpicture}
\]
where $\overline{F_*}$ is an automorphism of $\mathbb{P}^1$ that sends the branch locus of $\pi_a$ to the branch locus of $\pi_{a'}$.
These branch loci are precisely the zero loci of $a$ and $a'$, and include the general points $p_1, \dots, p_n$. We conclude that $\overline{F_*}$ is the identity, and $a'$ is a nonzero multiple of $a$. Therefore the linear map 
$$
f: {\rm H}^0(\mathbb{P}^1,\omega_{\mathbb{P}^1}^{\otimes 2}(S)) \longrightarrow {\rm H}^0(\mathbb{P}^1,\omega_{\mathbb{P}^1}^{\otimes 2}(S))
$$ 
is multiplication by a nonzero constant. 
After rescaling if necessary, we may assume that $f$ is the identity.

\medskip

Let  $(E, {\bf v})\in \mathcal{M}_{\cA_{F}}^s$ be a general rank two parabolic vector bundle, and write $\varphi(E, {\bf v}) = (E', {\bf v}')$. 
As explained in the beginning of the section, and in view of Proposition~\ref{NCbu}, in order to prove the theorem, it suffices to show that the projectivized nilpotent cones $\P(N_E)$ and $\P(N_{E'})$ are isomorphic over $\P^1$. 

From the above discussion, we have the following commutative diagram:
 \[
  \begin{tikzpicture}[xscale=2.1,yscale=-1.5]
    \node (A0_0) at (0, 0) {$T^{*}_{E}\mathcal{M}_{\cA_{F}}^s$};
    \node (A0_2) at (2, 0) {$T^{*}_{E'}\mathcal{M}_{\cA_{F}}^s$};
    \node (A1_1) at (1, 1) {${\rm H}^0(\mathbb{P}^1,\omega_{\mathbb{P}^1}^{\otimes 2}(S)).$};
    \path (A0_0) edge [->,swap]node [auto] {$\scriptstyle{h_{E}}$} (A1_1);
    \path (A0_2) edge [->]node [auto] {$\scriptstyle{h_{E'}}$} (A1_1);
    \path (A0_0) edge [->]node [auto] {$\scriptstyle{d\varphi}$} (A0_2);
  \end{tikzpicture}
  \]
If $n\geq 6$ then, by Lemma~\ref{globgen}, there are surjective maps of vector bundles on $\P^1$
$$
T^{*}_{E}\mathcal{M}_{\cA_F}^s\times \P^1
\ \stackrel{\alpha}{\to} \ \SPEnd_0\big(E,\omega_{\mathbb{P}^1}(S)\big), \text{ and } \\
$$
$$
T^{*}_{E'}\mathcal{M}_{\cA_F}^s\times \P^1
\ \stackrel{\alpha'}{\to} \ \SPEnd_0\big(E',\omega_{\mathbb{P}^1}(S)\big).
$$
By Remark~\ref{NE_from_h}, the induced isomorphism 
$$
d\varphi: T^{*}_{E}\mathcal{M}_{\cA_F}^s\times \P^1 \longrightarrow T^{*}_{E'}\mathcal{M}_{\cA_F}^s\times \P^1
$$ 
maps $\alpha^{-1}({N}_{E})$ to $(\alpha')^{-1}({N}_{E'})$, and the kernel of $\alpha$ to the kernel of $\alpha'$.
Therefore it yields an isomorphism 
$$
\SPEnd_0\big(E,\omega_{\mathbb{P}^1}(S)\big)\cong \SPEnd_0\big(E',\omega_{\mathbb{P}^1}(S)\big)
$$ over $\P^1$ mapping ${N}_{E}$ to ${N}_{E'}$. We conclude that $\P(N_E)$ and $\P(N_{E'})$ are isomorphic over $\P^1$, as desired. 

\medskip

For $n=5$  the result follows from \cite[Proposition 1.9]{AC17}. 
\end{proof}

%
%


\section{{Models $\mathcal M_{\cA}$ that are small modifications of $\mathcal M_{\cA_F}$}}\label{bir}

In this section we determine the automorphism group of moduli spaces $\mathcal M_{\cA}$ that are small modifications of 
$\mathcal M_{\cA_F}$. The weight polytope $\Pi\subset \Delta$ consisting of weights $\cA$ for which this happens can be described after 
\cite{Ba91} and \cite{Mu05}. We note that \cite{Ba91} and \cite{Mu05} consider  moduli spaces of  rank $2$ parabolic vector bundles on $\P^1$ 
of degree $1$, while here we work with degree $0$.
So, in order to describe the weight polytopes that are relevant to our setting, we perform a reflection on the corresponding polytopes described in \cite{Ba91}. This reflection corresponds to an elementary transformation centered at one parabolic point, as explained in Paragraph~\ref{el}.

\begin{say}{\bf The polytopes $\Delta$ and $\Pi$.}\label{polys}
The vertices of the hypercube $[0,1]^{n}\subset \R^{n}$ are the points of the form 
$\xi_I=\big((\xi_I)_1, \dots, (\xi_I)_{n}\big)$, where 
$I\subset \{1, \dots, n\}$, $(\xi_I)_i=1$ if $i\in I$, and  $(\xi_I)_i=0$ otherwise.
The parity of the subset $I$ and the vertex $\xi_I$ is the parity of $|I|$.
For each subset $I\subset \{1, \dots, n\}$, consider the degree one polynomial in the $\alpha_i$'s:
$$
H_I \ := \ \sum_{j\not\in I} \alpha_j + \sum_{i\in I}(1-\alpha_i).
$$ 
For any  subset $J\subset \{1, \dots, n\}$, we have:
\stepcounter{thm}
\begin{equation} \label{eq:H_I(x_J)}
H_I(\xi_J) \ = \ \# (I^{^c}\cap J) + \# (J^{^c}\cap I).
\end{equation} 

Let $\Delta$ be the polytope generated by the even vertices of the hypercube.
From \eqref{eq:H_I(x_J)} we see that  $\Delta$ is defined by the following set of inequalities: 
$$
\Delta \  = \  \left\{ 
\begin{aligned}
& 0\leq \alpha_i \leq 1, \ & i\in \{1, \dots, n\} \\
&H_I\geq 1 , \ & I\subset \{1, \dots, n\} \text{ odd.}
\end{aligned}
\right.
$$

From \eqref{eq:H_I(x_J)} we also see that, for any vertex $\xi_I\in \Delta$, the hyperplane  
spanned by those vertices of $\Delta$ that are adjacent to $\xi_I$ is $(H_I = 2)$. 
Hence, the polytope $\Pi\subset \Delta$ defined in the introduction can be defined by the following set of inequalities:
$$
\Pi \  = \ \Delta \ \cap \  \Big( \ H_I\geq 2 \ \big| \  I\subset \{1, \dots, n\} \text{ even } \Big).
$$

More generally, we define a wall-and-chamber decomposition on $\Delta$ as follows. 
For each subset $I\subset \{1, \dots, n\}$, and each integer $k$ satisfying $2\leq k\leq \frac{n}{2}$
and $|I|\equiv k \mod 2$, consider the hyperplane 
$(H_I =  k)$.
Now take the complement in the interior of $\Delta$ of the hyperplane arrangement
\stepcounter{thm}
\begin{equation} \label{eq:MCD_Delta}
  \Big(\ H_I \ = \  k \ \Big)_{\ 2\leq k\leq \frac{n}{2}, \ |I|\equiv k \mod 2}
\end{equation} 
and consider its decomposition into connected components. Each connected component is called a \emph{chamber} of $\Delta$.

In \cite{Ba91}, Bauer proved that this  wall-and-chamber decomposition on $\Delta$ corresponds to the variation of GIT for the moduli spaces $\mathcal{M}_{\cA}$, and described the birational maps between models corresponding to different chambers. 
In particular, for $\frac{1}{n-2}<\varepsilon<\frac{1}{n-4}$ and $\cA_{\epsilon}=(1-\epsilon,\epsilon,\dots,\varepsilon)$, the moduli space ${{\mathcal M}}_{\cA_\varepsilon}$ is 
isomorphic to the blow-up $X^{n-3}_n$ of $\mathbb{P}^{n-3}$ at $n$ general points. This is known to be a  \textit{Mori dream space} (\cite[Theorem 1.3]{CT06}). In particular its effective cone $\Eff(X^{n-3}_n)$ comes with a \emph{Mori chamber decomposition}, and the chambers inside the movable cone $\Mov(X^{n-3}_n)\subset \Eff(X^{n-3}_n)$ can be identified with the ample cones 
of small $\Q$-factorial modifications of $X^{n-3}_n$ (\cite{HK00}). 
These are $\Q$-factorial projective varieties which are isomorphic to $X^{n-3}_n$ outside a subset of codimension at least two. 
Mukai realized in \cite{Mu05} that there is a linear projection 
$$
\phi \  : \ \R^{n+1} \ \cong \  N^1(X^{n-3}_n)    \   \longrightarrow \ \R^{n} 
$$
mapping the effective $\Eff(X^{n-3}_n)$ onto $\Delta$, so that the wall-and-chamber decomposition of  $\Delta$ is induced by the Mori chamber decomposition of $\Eff(X^{n-3}_n)$. 
More precisely, for an arbitrary weight $\cA\in \Delta$, let $\cC_\cA$ denote the subset of $\Delta$ consisting of weight vectors defining the same 
stability condition as $\cA$. Then the relative interior of $\cC_\cA$ is the image under $\phi$ of the ample cone of the moduli space $\mathcal M_{\cA}$. 
In particular, $\phi$ maps the anti-canonical class $-K_{X^{n-3}_n}$ to the central weight $A_{F}= \left(\frac{1}{2},\dots,\frac{1}{2}\right)$, and the movable cone  $\Mov(X^{n-3}_n)$ is mapped onto the  polytope $\Pi$.
The linear projection $\phi$ was made explicit in \cite[Section 3]{AM16}.
\end{say}

\begin{proof}[{Proof of Corollary~\ref{allweights}}]
Let $\cA$ be a weight vector in the interior of the polytope $\Pi$.
Recall from the introduction that an elementary transformation $el_R\in \textbf{El}$ defines an automorphism of $\mathcal{M}_{\cA}^{s}$
if and only if it is $\cA$-admissible, i.e., $\cA^R\in \cC_\cA$.
In particular, $\textbf{El}_{\cA}\ \subset \ \Aut\big(\mathcal{M}_{\cA}\big)$. 

As explained above, $\mathcal M_{\cA}$ is a small modifications of  $\mathcal M_{\cA_F}$. 
So any automorphism $\varphi\in \Aut\big(\mathcal{M}_{\cA}\big)$ induces a pseudo-automorphism $\varphi_F$ of $\cM_{\cA_F}$. This means that $\varphi_F$ is a birational automorphism of $\cM_{\cA_F}$ that restricts to an isomorphism on the complement of a subset of codimension $\geq 2$. Since $\cM_{\cA_F}$ is a Fano variety,  every pseudo-automorphism of $\cM_{\cA_F}$ is in fact an automorphism, and hence $\varphi_F\in \textbf{El}$. 
We conclude that $\varphi$ is induced by an elementary transformation, $\varphi\in \textbf{El}_{\cA}$.
\end{proof}

\begin{Remark}
If $n\ge 6$, it follows from Corollary~\ref{allweights} that $\Aut(X_n^{n-3}) = \{Id\}$. In fact, taking $\cA_{\epsilon}$ with $\epsilon = \frac{1}{n-3}$ we can check that there is no $\cA$-admissible elementary transformation other than identity. Besides that, for appropriate choices of weights, there are small modifications $\mathcal M_{\cA}$ of $X_n^{n-3}$ having intermediate automorphism group.
\end{Remark}

%
%


\section{Moduli of involutional vector bundles}\label{hyp}

Let $n = 2g+2\geq 6$ be an even integer, and fix  $p_1, \dots, p_n\in \P^1$ general points.
Let $\pi: C \longrightarrow \P^1$ be the $2:1$ cover branched over the points $p_1, \dots, p_n$, set $c_i=\pi^{-1}(p_i)\in C$, and denote by $i:C \longrightarrow C$ the hyperelliptic involution. An involutional vector bundle $(E,j)$ on $C$ is an $i$-invariant vector bundle $E$, together with a lift $j: E \longrightarrow E$ of the involution $i$ to $E$.
We denote by ${\mathcal{M}}_{C/\P^1}^{inv}$ the moduli space of rank two semistable involutional vector bundles on $C$ with trivial determinant, and such that $\tr(j_{c_i})=0$ for $1 \leq i \leq n$ (see for instance \cite[Section 2]{Ab04}). 
Forgetting the lift $j: E \longrightarrow E$ yields a $2:1$ morphism 
$$
\pi\ : \ {\mathcal{M}}_{C/\P^1}^{inv} \ \longrightarrow \ \mathcal{S}
$$  
onto an irreducible component $\mathcal{S}$ of the moduli space of $i$-invariant rank two semistable vector bundles on $C$  with trivial determinant.

The Kummer variety of $C$ is 
$$\Kum(C) = \frac{\Jac(C)}{\iota},$$
where $\iota: \Jac(C)\longrightarrow\Jac(C)$ is the involution induced by $i: C \longrightarrow C$.
It naturally embeds in $\mathcal{S}$ via the map
$$
\begin{array}{ccc}
\Kum(C)& \longrightarrow & \mathcal{S}\\
   L & \longmapsto & L\oplus i^{*}L.
\end{array}
$$
By \cite[Theorem 2.1]{Ku00}, the double cover $\pi: {\mathcal{M}}_{C/\P^1}^{inv}  \longrightarrow  \mathcal{S}$ is branched over the Kummer variety $\Kum(C)\subset\mathcal{S}$. We denote by 
$$
\eta \ :\ {\mathcal{M}}_{C/\P^1}^{inv} \  \longrightarrow \ {\mathcal{M}}_{C/\P^1}^{inv}
$$
the involution of ${\mathcal{M}}_{C/\P^1}^{inv}$ induced by $\pi: {\mathcal{M}}_{C/\P^1}^{inv}  \longrightarrow  \mathcal{S}$.

As in the previous sections, we denote by $\cM_{\cA_{F}}$ the moduli space of rank two parabolic vector bundles with trivial determinant on $\big(\P^1, S\big)$ which are semistable with respect to the central weight $\cA_{F} = \left(\frac{1}{2},\dots, \frac{1}{2}\right)$.
By \cite[Proposition 1.2]{Bh84} there is an isomorphism 
\stepcounter{thm}
\begin{equation}\label{iso}
\cM_{\cA_F}\ \cong \ {\mathcal{M}}_{C/\P^1}^{inv}.
\end{equation}
This map is obtained by pulling back  parabolic vector bundles on $\P^1$ to $C$, performing an elementary transformation centered at the points $c_1,\dots,c_n$, and then twisting by an appropriate line bundle.

\begin{Proposition}\label{proInv}
Let the notation be as above. Then there is a splitting exact sequence
$$0 \ \rightarrow \  \frac{\mathbb{Z}}{2\mathbb{Z}} \cong \left\{ Id,\eta\right\} \ \rightarrow \  \Aut({\mathcal{M}}_{C/\P^1}^{inv}) \ 
\rightarrow \  \Aut(\mathcal{S},\Kum(C))\rightarrow 0,$$  
where $\Aut(\mathcal{S},\Kum(C))$ denotes the group of automorphisms of $\mathcal{S}$ stabilizing $\Kum(C)$.
\end{Proposition}

\begin{proof}
Let $X$ be the blow-up of $\mathbb{P}^{2g-1}$ at $p_1, \dots, p_{2g+2}$, and denote by $E_i$ the exceptional divisor over $p_i$.
By \cite{Ba91} and \cite{Mu05}, there is a small birational modification 
$$
f\ :\ X \ \dashrightarrow \ {\mathcal{M}}_{\cA_{F}}\ \cong \ {\mathcal{M}}_{C/\P^1}^{inv} \ ,
$$
which is defined by the linear system $\big|m(-K_X)\big|$ for $m\gg 1$.

Let $\mathcal{L}$ be the linear system on $\mathbb{P}^{2g-1}$ of degree $g$ hypersurfaces  having multiplicity at least $g-1$ at the $p_i$'s,
and denote by $ \mathcal{L}_X$ the induced linear system on $X$. Then
\stepcounter{thm}
\begin{equation}\label{can}
\displaystyle{-K_{X} \sim 2gH-\sum_{i=1}^{2g+2}(2g-2)E_i\sim 2 (gH-\sum_{i=1}^{2g+2}(g-1)E_i )\sim 2  \mathcal{L}_X.}
\end{equation}
By \cite[Theorem 2.1]{Ku00} the rational map $f_{\mathcal{L}}$ induced by $\mathcal{L}$ is generically $2:1$, dominant onto $\mathcal{S}$, and makes the following diagram commute 
  \[
  \begin{tikzpicture}[xscale=3.5,yscale=-1.5]
    \node (A0_0) at (0, 0) {$X$};
    \node (A0_1) at (1, 0) {${\mathcal{M}}_{C/\P^1}^{inv}$};
    \node (A1_0) at (0, 1) {$\mathbb{P}^{2g-1}$};
    \node (A1_1) at (1, 1) {$\mathcal{S}.$};
    \path (A0_0) edge [->,dashed]node [auto] {$\scriptstyle{f}$} (A0_1);
    \path (A1_0) edge [->,dashed]node [auto] {$\scriptstyle{f_{\mathcal{L}}}$} (A1_1);
    \path (A0_1) edge [->]node [auto] {$\scriptstyle{\pi}$} (A1_1);
    \path (A0_0) edge [->]node [auto] {$\scriptstyle{}$} (A1_0);
  \end{tikzpicture}
  \]
Since  $f$ is a small birational modification, \eqref{can} implies that $\pi :  {\mathcal{M}}_{C/\P^1}^{inv}  \longrightarrow \mathcal{S}$ is 
defined by the linear system $|\mathcal{L'}|$, where $2\mathcal{L'}\sim -K_{{\mathcal{M}}_{C/\P^1}^{inv}}$. 
In particular, any automorphism of ${\mathcal{M}}_{C/\P^1}^{inv}$ preserves the fibers of $\pi$, and hence descends to an automorphism of $\mathcal{S}$ stabilizing $\Kum(C)$, the branch locus of $\pi$. This gives a group homomorphism
$$
\Aut({\mathcal{M}}_{C/\P^1}^{inv}) \ \longrightarrow \  \Aut(\mathcal{S},\Kum(C)).
$$
Any automorphism in $\Aut(\mathcal{S},\Kum(C))$  lifts to an automorphism of ${\mathcal{M}}_{C/\P^1}^{inv}$. 
Furthermore, if $\varphi\in \Aut({\mathcal{M}}_{C/\P^1}^{inv})$ is a nontrivial automorphism that descends to the identity, 
then $\varphi$ must switch the two points on a general fiber of $\pi$, i.e, $\varphi=\eta$, yielding the stated exact sequence.
\end{proof}

\begin{Remark} \label{aut(k,c)}
By Theorem~\ref{main_intro} and \eqref{iso}, ${\mathcal{M}}_{C/\P^1}^{inv}\cong  \left(\frac{\mathbb{Z}}{2\mathbb{Z}}\right)^{2g+1}$. Proposition~\ref{proInv} yields
$$
\Aut(\mathcal{S},\Kum(C))\cong \left(\frac{\mathbb{Z}}{2\mathbb{Z}}\right)^{2g}.
$$
On the other hand, tensoring by a $2$-torsion line bundle on $C$ induces an automorphism of ${\mathcal{M}}_{C/\P^1}^{inv}$.
Therefore, $\Aut(\mathcal{S},\Kum(C))$ can be naturally identified with the group of $2$-torsion points of $\Jac(C)$.  
\end{Remark}

\begin{Remark}
When $n=6$, the description of  ${\mathcal{M}}_{C/\P^1}^{inv}$ is classical. In this case, $C$ is a curve of genus $2$,  $\mathcal{S}\cong \P^3$, and $\Kum(C)\subset\mathbb{P}^3$ is the classical Kummer surface. It is a quartic surface whose singular locus consists of $16$ singular points of type $A_1$. Remark~\ref{aut(k,c)} above recovers the group of automorphisms of $\mathbb{P}^3$ stabilizing $\Kum(C)$: 
$$
\Aut(\mathbb{P}^3,\Kum(C))\cong \left (\frac{\mathbb{Z}}{2\mathbb{Z}}\right )^{4}.
$$
\end{Remark}

\begin{Remark}
When $n = 2g+1\geq 5$ is odd, there is a similar isomorphism as in \eqref{iso}. 
In this case, in addition to the parabolic points $p_1, \dots, p_n\in \P^1$, we pick an extra general point $p_{2g+2}\in\P^1$.
As before, we let $\pi: C \longrightarrow \P^1$ be the $2:1$ cover branched over the points $p_1, \dots, p_{2g+2}$, and set $c_i=\pi^{-1}(p_i)\in C$.
We denote by ${\mathcal{M}}_{C/\P^1}^{inv}$ the coarse moduli space of rank two semistable involutional vector bundles on $C$ with determinant $\mathcal{O}_C(c_{2g+2})$, such that $\tr(j_{c_i})=0$ for $1 \leq i \leq n$, and $j_{c_{2g+2}}=(-1)^{(g-1)}Id$. 
Then 
$$
{\mathcal{M}}_{\cA_F}\ \cong \ {\mathcal{M}}_{C/\P^1}^{inv}.
$$
By Theorem~\ref{main_intro}, their automorphism groups are isomorphic to $\left(\frac{\mathbb{Z}}{2\mathbb{Z}}\right)^{2g}$. 
As before, they can be naturally identified with the group of $2$-torsion points of $\Jac(C)$.
\end{Remark}

\bibliographystyle{amsalpha}
\bibliography{Biblio}
\end{document}